\documentclass[a4paper, 12pt]{article}

\usepackage{amsthm, amsmath, amssymb, mathrsfs, multirow, url, subfigure}
\usepackage{graphicx} 
\usepackage{ifthen} 
\usepackage{amsfonts}
\usepackage{systeme}
\usepackage[usenames]{color}
\usepackage{fullpage}
\usepackage{hhline}
\usepackage{blkarray}
\usepackage{cite}
\usepackage{mathtools}

\usepackage[colorlinks,citecolor=blue]{hyperref}

\theoremstyle{plain} 
\newtheorem{theorem}{Theorem}
\newtheorem{corollary}{Corollary}

\theoremstyle{definition}
\newtheorem{defn}{Definition}

\theoremstyle{remark} 
\newtheorem{remark}{Remark}

\renewcommand{\phi}{\varphi}

\title{On The Study Of D-Optimal  Saturated  Designs  For Mean,  Main Effects and $F_1$-Two-Factor Interactions For $2^k$-Factorial Experiments}
\author{Francois Domagni \qquad A. S.  Hedayat \qquad Bikas Kumar Sinha  \\
Department of Mathematics, Statistics, and Computer Science \\
University of Illinois at Chicago \\
{\tt fdomag2@uic.edu, hedayat@uic.edu, bikassinha1946@gmail.com  } 
}
\date{\today}

\begin{document}

\maketitle

\begin{abstract}

The goal of this paper is to develop methods for the construction of saturated designs that include the mean, main effects and the two-factor interactions of one factor with a subset of  the remaining factors. If one factor is interacting with all the remaining factors  give a method for the  construction of  a d-optimal saturated design. If one factor is interacting with proper subset of the remaining factor we discuss the saturated d-optimal design  for specific cases. 

\smallskip

\end{abstract}

\section{Introduction}
\label{S:intro}

A  saturated design (SD) in a two-level factorial  experiment  is a design with the minimum number of runs that ensures the unbiased estimation of  the effects and interactions of interest given the remaining parameters are negligible. The number of runs $n$ retained in a SD is equal to the total number of parameters of interest. Thus  a saturated design  matrix is a square non-singular matrix of order $n$ with entries from $\lbrace -1, 1\rbrace$ that is chosen so as to satisfy the conditions of the parameters of interest. The statistical model  retained in this paper for a SD is the regular linear model $Y = D\beta + \epsilon$, where $Y$ is the response variable and $\epsilon$ is the usual error term. The matrix $D$ is a saturated design matrix  for the  given vector parameter of interest $\beta$ . Once $D$ is chosen, the ordinary least square method (OLS) can  be used to obtain the unbiased estimation of the parameters of interest. That is $\hat{\beta} = (D^TD)^{-1}D^TY = D^{-1}Y$ . As a result of the estimator $\hat{\beta} = D^{-1}Y$ the  determinant of the Fisher information of a SD is maximal if the absolute value of the determinant of $D$ is maximal.  Saturated designs are one of the most important designs in practice. They are desirable to practitioners  mainly when the important effects and interactions  to be estimated are known beforehand . However it turns out that the construction of SD is not a trivial problem.  There has been a vast literature as well as ongoing investigation about the construction of  SD under certain conditions.  Hedayat  and Pesotan in  [\cite{Hedayat1}] and [\cite{Hedayat2}] have discussed how to construct a saturated design that includes the estimation of  the mean, the main effects and a selected number of second order interactions.
Furthermore various computer algorithms have been developed to search for SDs for two level factorial experiments. Some of which are SPAN, DETMAX. As a case in point, see  [\cite{Hedayat3}]. In this paper the  problem we propose to solve is two-fold. 
\begin{enumerate}

\item  In the first part we  propose methods for the construction of saturated  and d-optimal saturated  design matrices for the estimation of the mean , the main effects and  the two-factor interactions of one factor with the remaining factors. Specifically we consider a two-level factorial experiment with $k$  factors  $F_1, \cdots , F_k$  and we develop algorithms for the construction of a saturated design matrix as well as a saturated d-optimal design matrix that includes the estimation of the main effects $F_1, \cdots , F_k$, the $F_1-$two factor interactions $F_{12}, \cdots , F_{1k}$ and the mean that we denote by $F_0$. We define $\mathcal{G}(k,1)$ as the set of all such design matrices. 
 \item In the second part of the paper we propose methods for the construction of saturated design matrices for $k+n$ main effects $F_1, \cdots, F_k, F^e_1, \cdots, F^e_n$, the $F_1$-two factor interactions $F_{12}, \cdots , F_{1k}$  and the mean $F_0$. We define $\mathcal{G}_n(k,1)$ as the set that includes all such design matrices.   Then we study the d-optimal saturated design matrix for the specific cases of $n=1$ and $n = 2k$. 
 \end{enumerate}
 Our approach to the problem is to first show that any  element of $\mathcal{G}(k,1)$  and $\mathcal{G}_n(k,1)$ can be written of a specific  block matrix form.  Next, we prove  the absolute value of the determinant of such a  block matrix is  bounded above by some constant independently of the choice of the block matrix. We then come up with an algorithm for the construction of one such block matrix for which the absolute value of the determinant attains the upper bound.  Our  work is essentially based on the Maximal Determinant Problem of Hadamard that has gained a lot of attention in the last century. It  asks for the largest determinant value  of $\lbrace -1, +1\rbrace$-matrices $M_n$  of order $n$. The problem  has been studied extensively in the literature according to  $4$ different values of $n$ (   $n \equiv 0 \pmod{4}$,  $n \equiv 1 \pmod{4}$,  $n \equiv 2 \pmod{4}$ and  $n \equiv 3 \pmod{4}$  ).  An upper bound has been found for the determinant for each value of $n$. We recall that it is well known that  for $n = 1$ ,  $n=2$  and $n \equiv 0 \pmod{4}$ we have $ |det(M_n)| \leq n^{\frac{n}{2}}$. The question as to whether there always exists a $\lbrace -1, +1\rbrace$-matrix of order $n$ with $n \equiv 0 \pmod{4}$ which attains the upper bound $n^{\frac{n}{2}}$ goes back to  the  famous \textit{Hadamard conjecture} which states that the answer is yes. Even though the conjecture has not been proved for an arbitrary value of $n$, it is widely accepted to be true. Ehlich [\cite{Ehlich1}] and Wojtas [\cite{Wojtas}] independently showed that  for  $n \equiv 2 \pmod{4}$  we have  $ |det(M_n)| \leq (2n-2)(n-2)$  . 
 For $n \equiv 1 \pmod{4}$, we have  $ |det(M_n)| \leq \sqrt{ (n-1)^{n-1}(2n-1)  }$. This results is due to Ehlich [\cite{Ehlich1}] and Barba  [\cite{Barba}]. 
 Finally for $n \equiv 3 \pmod{4} $, Ehlich [\cite{Ehlich2} ] showed that  $ |det(M_n)| \leq  \sqrt{ (n-3)^{(n-s)}(n-3 + 4r)^u(n + 1 + 4r)^v(1-\frac{ur}{n-3 + 4r} - \frac{v(r + 1)}{n + 1 + 4r}  )} $,    
 where $s = 3$ for $n=3$ in which case it is assumed that   $(n-3)^{(n-s)} = 1$;  $s = 5$ for $n=7$; $s=6$ for $11\leq n \leq 59$; $s = 7$ for $n\geq 63 $.  The constants  $r$ , $v$ and $u$ are obtained as follows;  $r =\lfloor \frac{n}{s} \rfloor $,
 $v = n -rs $ and $u = s- v$.

 \section{Construction of d-optimal saturated  designs in $\mathcal{G}(k,1)$ } 
\subsection{Preliminaries}
In this section we  consider a two-level factorial experiment with $k$  factors $F_1, \cdots , F_k$. We investigate  the class of saturated design matrices for a vector parameter $\beta$ that includes the mean, the $k$ main effects and the second order interactions of  factor $F_1$ with the remaining factors $F_2 ,\cdots, F_k$ . 
More precisely, for such a problem there are $k$ main effects $F_1,\cdots , F_k$,  the mean $F_0$ and  $k-1$ second order interactions $F_{12},\cdots, F_{1k}$  . The total number of parameters to estimate is $2k$. A saturated design would therefore require $2k$ runs. To gain more intuition about the problem, we give an example about the particular case of $k =3$ as follows.  For $k=3$ the number of parameters to estimate is 6, namely,  $F_0, F_1, F_2, F_3, F_{12}, F_{13}$. It follows that  a saturated design would require 6 runs. Suppose we choose the candidate design with the  runs $\lbrace 
 (1,1,1) ;  (1,0,0) ; (1,0, 1) ; (0,0,1) ; (0,1,1) ;  (0,1,0) \rbrace $. 
 Then the candidate saturated design matrix would be  a square matrix of order $6$  that is obtained by converting the runs into the underlying design  matrix .  As illustrated below,  the first  matrix underlies the main effects  plus mean $F_1, F_2, F_3$ and   $F_0$ . The second matrix   underlies the second order interactions $F_{12}$ and  $F_{13}$ and  is obtained by taking the Sch\"ur product of $F_1$ with $F_2$ and $F_3$ respectively. The third matrix  is  the candidate saturated design  matrix obtained by combining the first and second matrices.  It is worth pointing out that for convenience we set the factors in the order $F_1, F_2, F_3, F_0$ so that the first  and last entries  of each run correspond to $F_1$ and $F_0$ respectively.
 \\
 \\
\[
\begin{blockarray}{cccc}
  F_1& F_2 &  F_3 &  F_0 \\
\begin{block}{[cccc]}
   + & \text{+}&\text{+} &\text{+} \\
    \text{+} & \text{-}& \text{-} &\text{+} \\
     \text{+} & \text{-}& \text{+} &\text{+}  \\  
 \text{-} & \text{-}&\text{+}   & \text{+} \\
  \text{-} & \text{+}& \text{+} & \text{+} \\
  \text{-} & \text{+}& \text{-} & \text{+} \\
 \end{block}
 \end{blockarray} 
 \Rightarrow     
  \begin{blockarray}{cc}
     F_{12}& F_{13}  \\
     \begin{block}{[cc]}
   + & \text{+}\\
    \text{-} & \text{-}\\
     \text{-} & \text{+} \\  
 \text{+} & \text{+}\\
  \text{-} & \text{-}\\
  \text{-} & \text{+}\\
    \end{block}
    \end{blockarray}
      \Rightarrow
\begin{blockarray}{cccccc}
  F_1& F_2 &  F_3 &  F_0 & F_{12} & F_{13}\\
\begin{block}{[ccc|ccc]}
  + & +&+ &+ & +&+\\
    +& -& - &+& -& -\\
     + & -& +&+& -&+\\ 
     \cline{1- 6}     
 - & -&+  & +& +&- \\
  - &+&+&+&-& -\\
  -& +& -& +& -& +\\
\end{block}
\end{blockarray}
  \]
  \\
 
It is important to observe that for the given candidate design matrix given above,  $F_1$ is of the form  $F_1 = \begin{bmatrix}
1_{3}\\-1_3
\end{bmatrix} $, where  The Sch\"ur product of $F_1$ by itself ($F_{11}$) yields  \\
\\
 $F_{11}  = \begin{bmatrix}
1_3*1_3\\-1_3*(-1_3)
\end{bmatrix} = \begin{bmatrix}
1\\ 1 \\ 1\\
\\
1\\1\\1
\end{bmatrix}  = F_{0}$ , $F_{12} = \begin{bmatrix}
1_{3}\\-1_3
\end{bmatrix} *F_2 =   \begin{bmatrix}
1_{3}\\-1_3
\end{bmatrix} *\begin{bmatrix}

1\\ -1\\ -1
\\
-1\\1\\1
\end{bmatrix} =  \begin{bmatrix}
1\\ -1 \\ -1
\\
\\

1\\-1\\-1
\end{bmatrix}  $ and  $F_{13} = \begin{bmatrix}
1_{3}\\-1_3
\end{bmatrix} *F_3 =   \begin{bmatrix}
1_{3}\\-1_3
\end{bmatrix}\begin{bmatrix}

1\\ -1\\ 1
\\
1\\1\\-1
\end{bmatrix}
=  \begin{bmatrix}
1\\ -1 \\ 1
\\
\\
-1\\-1\\1
\end{bmatrix}  $. \\
 Furthermore the Sch\"ur product of $F_1$ with $F_2 $ and $F_3$ leave the first 3 entries of $F_2$ and $F_3$ unchanged and negate the last 3 entries.
It turns out from the above observations that the candidate saturated design can be written as:

\[
\begin{blockarray}{cccccc}
  F_1& F_2 &  F_3 &  F_{11} & F_{12} & F_{13}\\
\begin{block}{[ccc|ccc]}
  + & +&+ &+ & +&+\\
    +& -& - &+& -& -\\
     + & -& +&+& -&+\\ 
     \cline{1- 6}     
 - & -&+  & +& +&- \\
  - &+&+&+&-& -\\
  -& +& -& +& -& +\\
\end{block}
\end{blockarray}
  = \begin{bmatrix}
    M & M\\
   - N & N
    \end{bmatrix} 
    \],

  where $M =  \begin{bmatrix}
   \text{+} & \text{+}&\text{+} \\
    \text{+} & \text{-}& \text{-} \\
     \text{+} & \text{-}& \text{+} \end{bmatrix}  $  and $N =\begin{bmatrix}
   \text{+} & \text{+}&\text{-} \\
    \text{+} & \text{-}& \text{-} \\
     \text{+} & \text{-}& \text{+} \end{bmatrix}  $.
  \\
  \\
  \begin{remark}
   A few remarks can be made as follows;
  \begin{enumerate}
  \item The mean $F_0$ can be written as the Sch\"ur product of $F_1$ by itself. This simple fact will be crucial in the theorems we  develop in the upcoming section. 
  \item For any choice of candidate saturated design the corresponding candidate saturated design matrix is necessarily of the form $\begin{bmatrix}
    M & M\\
    N & -N
    \end{bmatrix}$ as shown above.  In the example  $F_1$ has as many $+1$ entries as $-1$ entries which means $F_1$ is balanced.  Therefore  $M$ and $N$ are square matrices of order $k$. 
   \item   The candidate  design matrix as displayed above  will be a valid design matrix  if it is a non-singular matrix. We shall see in the remainder of this paper that in general  a candidate design matrix is  a valid design matrix  if and only if the design is chosen so that $F_1$ is balanced and that $M$ and $N$ are non-singular matrices. 
  \end{enumerate}
  \end{remark}

 \subsection{ Construction of saturated and saturated d-optimal design matrices in $\mathcal{G}(k,1)$}
  In the remaining of this section we explore the the construction of a d-optimal design matrices for mean, main effects and the $F_1$-second-order- interactions from a general perspective.  We assume without loss of generality that  the vector parameter of interest is of the form    $\beta = [ F_1,\cdots, F_k, F_0,  F_{12}, \cdots, F_{1k}]^T $ . For convenience we make the following definitions.
 \begin{defn} 
 We make the following definitions;
 \begin{enumerate}

\item  We define $\mathcal{G}(k, 1)$ to be the set of all the saturated design matrices that ensure the unbiased estimation of the vector parameter of interest  $\beta$ .  We purposely use the notation $\mathcal{G}(k, 1)$ to indicate that the vector parameter of interest $\beta$ includes  the $k$ main effects , the mean ,   and all the  $F_1$-second-order-interactions. 
\item We  define  $g(k, 1)$ to be an element of $\mathcal{G}(k, 1)$. It is worth pointing out that $g(k, 1)$ is a non-singular matrix of order $2k$ with entries from $\lbrace -1, 1\rbrace$ that satisfies the condition of the parameter $\beta$.
\item We define $\mathcal{M}_{k}\lbrace -1,1\rbrace $  as the set of non-singular matrices of order $k$ with entries from $\lbrace -1, 1\rbrace$ for which the first  column is the vector $1_k$.
 \item We define $\Theta_k$ to be the maximal value of the absolute value of the  determinant of matrices in $\mathcal{M}_k\lbrace -1, +1 \rbrace$.
   \end{enumerate}
   
  \end{defn}
  The factor $F_1$ plays a key role in the construction of a saturated design for the vector parameter $\beta $ as specified above because it is the only factor that interacts with all the remaining factors. Therefore we define the factor $F_1$ as the pivot factor. Since the entries of $F_1$ takes values from $\lbrace -1, 1\rbrace $ we assume without loss of generality that $F_1$ is of the form $F_1 = \begin{bmatrix}
  1_{f_+}^T & -1_{f_-}^T
  \end{bmatrix}^T$,  where $ f_+ $ and $f_-$  are respectively the frequencies of $1$ and $-1$ entries  in the vector $F_1$ with $ f_+ + f_- = 2k$. 
   For convenience   we write $F_2, \cdots,  F_k$ as block vectors 
   $ F_2 =\begin{bmatrix}
m_2^T & n_2^T
\end{bmatrix}^T$
$, \cdots , $
 $F_k = \begin{bmatrix}
m_k^T & n_k^T
\end{bmatrix} ^T$  ,
where $m_2,\cdots , m_k$ are vectors of length $f_+$ and $n_2,\cdots , n_k$ are vectors of length $f_-$  with entries from $\lbrace -1, 1\rbrace $.  We enumerate the following key observations.
\\
  \begin{enumerate}
\item The $F_1$-second-order-interactions $F_{12}, \cdots, F_{1k}$ are  obtained by the Sch\"ur product of $F_1$  with $F_{2}, \cdots, F_{k}$ as follows:
\\
  $ F_{12} =\begin{bmatrix}
(1_{f_+}*m_2)^T &
(-1_{f_-}*n_2)^T
\end{bmatrix}^T  = \begin{bmatrix}
m_2^T &
-n_2^T
\end{bmatrix}^T$\\
$\vdots $\\
 $F_{1k} = \begin{bmatrix}
(1_{f_+}*m_k)^T &
(-1_{f_-}*n_k)^T
\end{bmatrix}^T =  \begin{bmatrix}
m_k^T &
-n_k^T
\end{bmatrix}^T$.  

\item The mean $F_0$ which is a  $1_{2k}$ column vector can be written as $F_0 = \begin{bmatrix}
1_{f_+}^T &
1_{f_-}^T
\end{bmatrix}^T = \begin{bmatrix}
(1_{f_+}*1_{f_+})^T &
(-1_{f_-}*(-1_{f_-})^T
\end{bmatrix}^T  =F_1*F_1$ .  That is the mean $F_0$ can be obtained by the Sc\"ur product of $F_1$ with itself. 
\end{enumerate}
 By preserving the order  in which  the parameters in the vector $\beta = [ F_1,\cdots, F_k, F_0,  F_{12}, \cdots F_{1k}]^T $ appear,   each element of $\mathcal{G}(k, 1)$  can be written as :
\\
$\begin{bmatrix}
1_{f_+} & m_2 &\cdots & m_k & & 1_{f_+} & m_2 &\cdots & m_k \\
-1_{f_-} & n_2 &\cdots & n_k & & 1_{f_-} & -n_2 &\cdots & -n_k
\end{bmatrix} =   \begin{bmatrix}
  M   &      M\\
  
 -N    &      N
\end{bmatrix}  $ , \\
\\
  where $ M =\begin{bmatrix}
1_{f_+} & m_2 &\cdots & m_k 
\end{bmatrix} $  and $ N =\begin{bmatrix}
1_{f_-} & n_2 &\cdots & n_k 
\end{bmatrix} $ with dimensions $f_+ \times k $ and $f_- \times k $ respectively.
Thus  each element of $\mathcal{G}(k, 1)$ is necessarily on the block matrix form  $\begin{bmatrix}
  M   &      M\\
  
 -N    &      N
\end{bmatrix}  $.  Now just because we have the block matrix form $\begin{bmatrix}
  M   &      M\\
  
 -N    &      N
\end{bmatrix}  $ doesn't mean that we have obtained an element of $\mathcal{G}(k, 1)$. The question one may ask is \text{''} What are the necessary and sufficient conditions on  the matrix $\begin{bmatrix}
  M   &      M\\
  
 -N    &      N
\end{bmatrix} $ to be an element of $\mathcal{G}(k, 1)$ ? \text{''}.
  
 Our goal in what follows  is to provide  necessary and sufficient conditions to construct an element of $\mathcal{G}(k, 1)$.
In the theorem below we provide  the necessary and sufficient conditions to construct an element of $\mathcal{G}(k, 1)$.

\begin{theorem}
\label{thrm}
Let $D$ be a design matrix then $D$ is an element of $\mathcal{G}(k, 1)$ if and only if it can be written as $D =   \begin{bmatrix}
  M_1   &      M_1\\
  
 -N_1    &      N_1
\end{bmatrix} $  where $M_1$ and $N_1$ are elements of $\mathcal{M}_{k}\lbrace -1,1\rbrace $.

\end{theorem}
 
 \begin{proof}
We have seen that any element of $\mathcal{G}(k, 1)$ is necessarily on the form 
$
  F = \begin{bmatrix}
  M   &      M\\
  
 -N    &      N
\end{bmatrix}
$,
 where   $M $ and  $N $ are $\lbrace -1, 1\rbrace$-matrices of dimensions $f_+\times k$ and $f_-\times k$ respectively.  We will first show that  if $f_+ \neq f_-$ then the matrix $F$  is a singular matrix. In that  case F is not an element of $\mathcal{G}(k, 1)$ . We then show that $M$ and $N$ have to be both non-singular matrices of order $k$ for $F$ to be an element of $\mathcal{G}(k, 1)$ .

 \begin{enumerate}
 
 \item Assume without loss of generality that $f_+ > k$. 
 Then since $M$ is of dimensions
$f_+\times k$  we have $rank(M)$ is at most $k$. Therefore the rows of $M$ that we define as  $m_1^T, \cdots, m_{f_+}^T$ are linearly dependent. We may assume without  loss of generality that $m_1^T$ is linearly dependent on $m_2^T, \cdots, m_{f_{+}}^T$ ,  so that $m_1 = \sum_{i=2}^{f_+} c_im_i$  with some   $c_i\neq 0$, $\quad 2\leq i \leq f_{+}$. This implies that
$\begin{bmatrix}
m_1\\
m_1
\end{bmatrix} =  \sum_{i=2}^{f_+} c_i \begin{bmatrix}
m_i\\
m_i
\end{bmatrix} $.  It  means that the  rows $\begin{bmatrix}
m_1^T & m_1^T
\end{bmatrix} , \cdots , \begin{bmatrix}
m_{f_+}^T & m_{f_+}^T
\end{bmatrix}$ of $F$ are linearly dependent,  which would make $F$ a singular matrix.  In a similar manner one can show that   if  $f_{-}> k$ then $F$ is a singular matrix .  Thus it turns out  that $f_{-} = f_{+} = k$ is a necessary condition on  $F$ to be non-singular . 
It follows that  any element $F$ of $\mathcal{G}(k, 1)$ is on the form $
  F = \begin{bmatrix}
  M   &      M\\
  
 -N    &      N
\end{bmatrix}
$, 
 where   $M $ and  $N $ are $\lbrace -1, 1\rbrace$-matrices of order $k$.  \\
 
Now  If the matrix $M$ is singular  the rows of $F$ would be linearly dependent and F would be a singular matrix by analogy of the  argument above. By the  same argument  if $N$ is singular, F would be a singular matrix. 
 
 \item Now suppose both $M$ and $N$ are non-singular matrices, that is  $M$ and $N$ are elements of $\mathcal{M}_{k}\lbrace -1,1\rbrace $. Then $det(F) = det \begin{bmatrix}
  M   &      M\\
  
 -N    &      N
\end{bmatrix} = det (N)det(M+ MN^{-1}N ) = 2^kdet(N)det(M)\neq 0 $
It follows that $F$ is an element of $\mathcal{G}(k, 1)$ if and only if $ F = \begin{bmatrix}
  M   &      M\\
  
 -N    &      N
\end{bmatrix}
$,  where $M$ and $N$  elements of $\mathcal{M}_{k}\lbrace -1,1\rbrace $.
 \end{enumerate}
\end{proof}
 \begin{corollary}
 \label{corol}
 A design matrix $D^*$ is a d-optimal saturated design in $\mathcal{G}(k, 1)$ if and only if it can be written as $D^* =   \begin{bmatrix}
  M_1^*   &      M_1^*\\
  
 -N_1^*    &      N_1^*
\end{bmatrix} $  where $M_1^*$ and $N_1^*$ are elements of $\mathcal{M}_{k}\lbrace -1,1\rbrace $ with maximal determinant. Furthermore $|det(D^*)| = 2^k\Theta_k^2$

 \end{corollary}
 
 \begin{proof}
 By Theorem (\ref{thrm}) for any element $D$ of $\mathcal{G}(k, 1)$, $ det(D) = 2^kdet(N)det(M)$ for some $N$ and $M$ elements $\mathcal{M}_{k}\lbrace -1,1\rbrace $. This determinant is maximal  when both $M$ and $N$ have maximal determinant in $\mathcal{M}_{k}\lbrace -1,1\rbrace $.
 \end{proof}
 

 
 \subsection{Algorithm for the construction of an element of $\mathcal{G}(k,1)$}
 \label{goptimal1}
 We use Theorem (\ref{thrm}) and Corollary (\ref{corol}) to develop  an algorithm for the construction of a saturated and a d-optimal saturated design matrix of $\mathcal{G}(k,1)$.
 
 \begin{itemize}
\item \textbf{Step 1} : 
 Select two matrices $M$ and $N$ from $\mathcal{M}_{k}\lbrace -1,1\rbrace $ (For  a d-optimal design select the  matrices $M$ and $N$ with maximal absolute value of determinant)
 \if false\item \textbf{Step 2}: Fold over the design matrix $M$ or (select another $N$ from $\mathcal{M}_{k}\lbrace -1,1\rbrace $) to obtain  $K = \begin{bmatrix} 
  M   \\
  
 -M  
\end{bmatrix}$ or ($K = \begin{bmatrix} 
  M   \\
  
 -N 
\end{bmatrix}$     ) of dimensions $2k\times k$ 
\fi
\item \textbf{Step2}: The design matrices $D_1 = \begin{bmatrix} 
  M  & M\\
  
 -M  & M
\end{bmatrix} $ and $D_2 = \begin{bmatrix} 
  M  & M\\
  
 -N  & N
\end{bmatrix} $ obtain through the above steps are  saturated design matrices for the estimation of the mean $F_0$, the  $k$  main effects $F_1,\cdots, F_k$ and the interactions $ F_{12}, \cdots, F_{1k}$. 
 $D_1$ is a d-optimal design matrix in $\mathcal{G}(k, 1)$ if $ |det(M)|$ is maximal in $\mathcal{M}_{k}\lbrace -1,1\rbrace $.  $D_2 $  is a d-optimal design matrix in $\mathcal{G}(k, 1)$  if  both $|det(M)|$ and $|det(N)|$ have  maximal  determinant in $\mathcal{M}_{k}\lbrace -1,1\rbrace $ .  

 \end{itemize}

\section{Construction of d-optimal saturated  designs in $\mathcal{G}_n(k,1)$  }
 In the previous section we developed theorems and algorithms for the construction of saturated design matrices that are elements of $\mathcal{G}(k,1)$.  In this section we consider a \\2-level factorial experiment with $k+n$ factors that we denote $F_1,\cdots, F_k$ for the $k$ factors and $F_1^e,\cdots, F_n^e$ for the remaining $n$ factors. Our goal here  is to provide algorithms for the construction of a saturated and a saturated d-optimal  design matrices that include the parameters $F_0  , F_1,\cdots, F_k , F_{12},\cdots, F_{1k}$ and the extra main effects  
$F_1^e,\cdots, F_n^e$.  We define $\mathcal{G}_n(k,1)$ as the set of such  saturated design matrices  and $g_n(k,1)$ to be an element of $\mathcal{G}_n(k,1)$. As an example suppose in a $2^5$-factorial experiment the investigator is interested in finding a saturated design matrix  for the estimation of the mean $F_0$, all the main effects $F_1, F_2, F_3, F_4, F_5$ and the two factor interactions $F_{12} , F_{13}$. For this particular problem  one  could rearrange the parameters of interest  as $F_0, F_1, F_2, F_3, F_{12}, F_{13}$ and $F_4, F_5$ so that it becomes  a problem of finding an element of $\mathcal{G}_n(k,1)$ with $k=3$ and $n=2$ such that $F_1^e = F_4$ and $F_2^e = F_5$. Thus the parameters of  interest may be written as $F_1, F_2, F_3, F_0, F_{12}, F_{13} $ and $F_1^e ,F_2^e$. It is worth pointing out that if there was no extra  main effects $F_1^e , F_2^e$ then the parameters of interest would be $F_1, F_2, F_3, F_0, F_{12}, F_{13} $ . The problem would just boil down to finding an element of $\mathcal{G}(3,1)$ which we have discussed extensively in the previous section. But for the problem at hand two extra parameters need to be included in the design .  
Our approach to construct an element of $\mathcal{G}_n(k,1)$ would be to first construct an element of $\mathcal{G}(k,1)$  and an element of $\mathcal{M}_{n}\lbrace -1, 1\rbrace $ . Then we try to combine the two matrices constructed in a way  to form an element of $\mathcal{G}_n(k,1)$.  From now onward  we assume that  the vector parameter of interest  is of the form $\beta =\begin{bmatrix}
F_1 & \cdots & F_k & F_{11} \cdots F_{1k} & F_1^e \cdots F_n^e
\end{bmatrix}$ with the parameters appearing in that order. Therefore with the order preservation of the parameters in $\beta$, it is straightforward to observe  that any element of  $\mathcal{G}_n(k,1)$ can be written on the block matrix  form  $\begin{bmatrix}
g(k,1) & -K_1\\
K_2 & M_n
\end{bmatrix}$ , where $g(k,1)$ is an element  of $\mathcal{G}(k,1)$, $M_n$ is an element of $\mathcal{M}_{n}\lbrace -1, 1\rbrace $, $K_1$ is an $2k\times n$ matrix with entries from $\lbrace -1, 1\rbrace$,  and $K_2$ is an $n\times 2k$ matrix and its rows are of the form $\begin{bmatrix}
r^T & r^T
\end{bmatrix}$ or $\begin{bmatrix}
-r^T & r^T
\end{bmatrix}$,  where $r$ is  a vector of length $k$ with its first  entry being $+1$  and its remaining entries are from $\lbrace -1, 1\rbrace$.
  For convenience  we make the following definition . 
\begin{defn}
\item  Let $A$ be a matrix,  we  define $\mathcal{R}[A]$ to be the set of the rows of $A$ . 

 \end{defn}
In the theorem below we give a method for the construction of an element of  $\mathcal{G}_n(k,1)$  for arbitrary $k$ and $n$. 
\begin{theorem}
\label{theorem2}
Let  $g(k,1)$  and $M_n$ be given elements of $\mathcal{G}(k,1)$ and   $ \mathcal{M}_{n}\lbrace -1, 1\rbrace $ respectively. Furthermore let $G$ and $V$ be two $n\times 2k$  and $2k\times n$ matrices respectively,  such that $\mathcal{R}[G]\subseteq \mathcal{R}[g(k,1)]$ and $\mathcal{R}[V]\subseteq \mathcal{R}[M_n]$. Then the matrix  $\begin{bmatrix}
g(k,1) & -V\\
G & M_n
\end{bmatrix}$ is a saturated  design matrix in  $\mathcal{G}_n(k,1)$ .
\end{theorem}
\begin{proof}
Our objective is to show that the block matrix  $\begin{bmatrix}
g(k,1) & -V\\
G & M_n
\end{bmatrix}$ satisfies the necessary  block matrix form of saturated design matrices in $\mathcal{G}_n(k,1)$. Then we show that it has non-zero determinant. 
\\
It is straightforward to observe that
 $\begin{bmatrix}
g(k,1) & -V\\
G & M_n
\end{bmatrix} $ as defined in the theorem satisfies the necessary block matrix form of elements in $\mathcal{G}_n(k,1)$.  In fact we know $g(k,1)$ is of the form  $\begin{bmatrix}
M & M\\
-N & N
\end{bmatrix} $. This implies that since $\mathcal{R}[G]\subseteq \mathcal{R}[g(k,1)]$ , every row of $G$ is of the form $\begin{bmatrix}
r^T & r^T
\end{bmatrix}$ or $\begin{bmatrix}
-r^T & r^T
\end{bmatrix}$,  where $r$ is  a vector of length $k$ with its first  entry being $+1$  and its remaining entries are from $\lbrace -1, 1\rbrace$. Therefore  $\begin{bmatrix}
g(k,1) & -V\\
G & M_n
\end{bmatrix} $ satisfies the necessary block matrix form of elements in $\mathcal{G}_n(k,1)$.
\\

Now we have
$det\Bigg (\begin{bmatrix}
g(k,1) & -V\\
G & M_n
\end{bmatrix}\Bigg ) = det(M_n)[ det\lbrace g(k,1) + VM_n^{-1}G\rbrace]$. We know $det(M_n) \neq 0$ since $M_n $ is an element of $\mathcal{M}_{n}\lbrace -1, 1\rbrace $. \\
It remains to show that  $det\lbrace g(k,1) + VM_n^{-1}G\rbrace\neq 0$. \\
\\
Let $M_n = \begin{bmatrix}
m_1^T\\
\vdots
\\
m_n^T
\end{bmatrix}
$ ,  then since $\mathcal{R}[V]\subseteq \mathcal{R}[M_n]$, we can write $V = \begin{bmatrix}
m_{j_1}^T\\
\vdots
\\
m_{j_{2k}}^T
\end{bmatrix}$ 
 \\where $j_i \in \lbrace 1,\cdots, n\rbrace $  and $i = 1, \cdots, 2k$.
\\
\\
 This implies that $VM_n^{-1} =\begin{bmatrix}
m_{j_1}^T\\
\vdots
\\
m_{j_{2k}}^T
\end{bmatrix}\begin{bmatrix}
m_1^T\\
\vdots
\\
m_n^T
\end{bmatrix}^{-1} = \begin{bmatrix}
\delta_{1j_1} & \cdots   &  \delta_{nj_1}\\
\vdots &\cdots & \vdots \\
\delta_{1j_{2k} } &  \cdots  &  \delta_{nj_{2k}}
\end{bmatrix} $ ,   where $\delta$ is the Kronecker delta function defined as
 $\delta_{pq} =  \begin{cases} 0, & \mbox{if } p\neq q \\ 1, & \mbox{if } p=q \end{cases}$. 
\\
\\
\\
Let $g(k,1) = \begin{bmatrix}
g_1^T\\
\vdots
\\
g_{2k}^T
\end{bmatrix}
$, then since $\mathcal{R}[G]\subseteq \mathcal{R}[g(k,1)]$ we can write $G= \begin{bmatrix}
g_{l_1}^T\\
\vdots
\\
g_{l_{n}}^T
\end{bmatrix}$ , 
 where $l_i \in \lbrace 1,\cdots, 2k\rbrace $  and $i = 1, \cdots, n$.
\\
\\
Therefore we have: $VM_n^{-1}G = \begin{bmatrix}
\delta_{1j_1} & \cdots   &  \delta_{nj_1}\\
\vdots &\cdots & \vdots \\
\delta_{1j_{2k} } &  \cdots  &  \delta_{nj_{2k}}
\end{bmatrix} \begin{bmatrix}
g_{l_1}^T\\
\vdots
\\
g_{l_{n}}^T
\end{bmatrix} = \begin{bmatrix}
g_{j_1}^T\\
\vdots
\\
g_{j_{2k}}^T
\end{bmatrix} $.
\\
\\
\\
 It follows that $  g(k,1) + VM_n^{-1}G = \begin{bmatrix}
g_1^T\\
\vdots
\\
g_{2k}^T
\end{bmatrix}
 +  \begin{bmatrix}
g_{j_1}^T\\
\vdots
\\
g_{j_{2k}}^T
\end{bmatrix}  $
\\
\\
It turns out that the matrix $  g(k,1) + VM_n^{-1}G $ is obtained  from $g(k,1)$ by adding  some of its rows to  itself. That means  that the determinant of $  g(k,1) + VM_n^{-1}G $ is proportional to the determinant of  $g(k,1)$ up to a non-zero constant. That is $det\lbrace g(k,1) + VM_n^{-1}G\rbrace \propto det\lbrace g(k,1)\rbrace $. Therefore the matrix $\begin{bmatrix}
g(k,1) & -V\\
G & M_n
\end{bmatrix}$ is a non-singular matrix. 

\end{proof}
\begin{remark}
Theorem (\ref{theorem2}) gives a general method for constructing an element of $\mathcal{G}_n(k,1)$.  Even though it does not directly address the problem  of constructing a saturated d-optimal design matrix in  $\mathcal{G}_n(k,1)$, it appears to be  useful if the interest of the experimenter is only the estimability of the vector  parameter $\beta =\begin{bmatrix}
F_1 & \cdots & F_k & F_{11} \cdots F_{1k} & F_1^e \cdots F_n^e
\end{bmatrix}$.
It turns out the problem of finding a saturated d-optimal design matrix for this particular $\beta$ is not trivial. We discuss the  d-optimality  problem below for some specific values of $k$ and $n$. 
\end{remark}. 
\subsection{ Saturated d-optimal  design matrix  in  $\mathcal{G}_{2k}(k,1)$\\ for  $k \equiv 0 \pmod{4}$ }
The saturated design matrices in  $\mathcal{G}_{2k}(k,1)$ are of order $4k$.  The corollary below 
\begin{corollary}
Suppose $g^*(k,1)$  is a d-optimal design matrix in  $ \mathcal{G}(k,1)$ with  $k \equiv 0 \pmod{4}$ . Then the matrix
 $\begin{bmatrix}
g^*(k,1) & -g^*(k,1)\\
g^*(k,1) & g^*(k,1)
\end{bmatrix}$ is a d-optimal design matrix in $\mathcal{G}_{2k}(k,1)$ .
\end{corollary}
\begin{proof}
  For $k = 4k'$ we have $g(k,1)$ is of the form $\begin{bmatrix}
M_{4k'} & M_{4k'}\\
-N_{4k'} & N_{4k'}
\end{bmatrix} 
$, where both $M_{4k'}$ and $ N_{4k'}$ are elements of $\mathcal{M}_{4k'}\lbrace -1, 1\rbrace $. 
$|det(g(k,1)|= | det \Bigg (\begin{bmatrix}
M_{4k'} & M_{4k'}\\
-N_{4k'} & N_{4k'}
\end{bmatrix} \Bigg )|
$ is  maximal  if $M_{4k}$ and $M_{4k}$ are both Hadamard matrices in $\mathcal{M}_{4k'}\lbrace -1, 1\rbrace $. This implies  the design matrix in $ \mathcal{G}(k,1)$  with maximal absolute value of the  determinant is a Hadamard matrix $g^*(k,1)$.  Since  $g^*(k,1)$ is a Hadamard matrix then  $g^*_{2k}(k,1) = \begin{bmatrix}
g^*(k,1) & -g^*(k,1)\\
g^*(k,1) & g^*(k,1)
\end{bmatrix}$ is also a Hadamard matrix in $\mathcal{G}_{2k}(k,1)$. The proof is complete. 

\end{proof}
\subsection{D-Optimal design in  $\mathcal{G}_1(k,1)$ }
Any element of $\mathcal{G}_1(k,1)$ is a matrix of order $2k+1$ which has an extra factor denoted $F^e_1$.
We show here how to construct  saturated  and d-optimal saturated  design matrices  in $\mathcal{G}_1(k,1)$ . 
\begin{theorem}
\label{theorem4}
Let $g_1(k,1)$ be an element of  $\mathcal{G}_1(k,1)$ then $|det(g_1(k,1))| \leq 2^k\Theta_k\Theta_{k+1}$. 
\end{theorem}
\begin{proof}
Our objective is to show that the absolute value of the determinant of any element of $\mathcal{G}_1(k,1)$ is bounded above by $2\Theta_k\Theta_{k+1}$. \\
Any design matrix element of $\mathcal{G}_1(k,1)$  can be written as 
$\begin{bmatrix}
g(k,1) &- c\\
v^T &     1
\end{bmatrix}$,  where $g(k,1)$ is an element of $\mathcal{G}(k,1)$  and  $c$ is  vector of length $2k$ with entries from $\lbrace -1, +1\rbrace$.  Furthermore $v$ is  a vector of length $2k$ with entries from $\lbrace -1, +1\rbrace$ such that  $v^T$ is of the form $v^T = \begin{bmatrix}
r^T & r^T
\end{bmatrix}$ or  $\begin{bmatrix}
-r^T & r^T
\end{bmatrix}$ ,  where $r$ is a vector of length $k$ and its first entry is $+1$.\\
\begin{enumerate}
 \item Assume that  $v^T = \begin{bmatrix}
r^T & r^T
\end{bmatrix}$.   We know $g(k,1)$ can be written as $g(k,1) = \begin{bmatrix}
  M   &      M\\
  \\
 -N    &      N
\end{bmatrix} $,  where $M$ and $N$ are  elements of $\mathcal{M}_k\lbrace -1, + 1\rbrace$. Therefore 

$| det\Bigg (\begin{bmatrix}
g(k,1) &- c\\
v^T &     1
\end{bmatrix}\Bigg ) |  = |det \lbrace g(k,1) + cv^T\rbrace |
 =|det \Bigg ( \begin{bmatrix}
  M   &      M\\
  \\
 -N    &      N
\end{bmatrix} + \begin{bmatrix}
c_1v^T\\
\vdots \\
c_{2k}v^T
\end{bmatrix}\Bigg ) | =
    |det \Bigg ( \begin{bmatrix}
  M   &      M\\
  \\
 -N    &      N
\end{bmatrix} + \begin{bmatrix}
\begin{bmatrix}
c_1r^T & c_1r^T
\end{bmatrix}\\
\vdots \\
\begin{bmatrix}
c_{2k}r^T & c_{2k}r^T
\end{bmatrix}
\end{bmatrix}\Bigg ) |  =
|det \Bigg ( \begin{bmatrix}
  M + E_1   &      M + E_1\\
  \\
 -N + E_2    &      N + E_2
\end{bmatrix}\Bigg ) |  $, 
where $E_1 =\begin{bmatrix}
c_1r^T \\
\vdots
\\
c_kr^T
\end{bmatrix} $  and $E_2 =\begin{bmatrix}
c_{k+1}r^T 
\\
\vdots
\\
c_{2k}r^T
\end{bmatrix} $.
\\
\\
  If  $M + E_1$ is a singular matrix then  the matrix $\begin{bmatrix}
  M + E_1   &      M + E_1\\
  \\
 -N + E_2    &      N + E_2
\end{bmatrix} $ is a singular matrix and its determinant would be zero (see the proof of Theorem (\ref{thrm}) for more details about these particular form of matrices ). 
Now if  $M + E_1$ is a non-singular matrix then we have \\
$|det\lbrace\begin{bmatrix}
  M + E_1   &      M + E_1\\
  \\
 -N + E_2    &      N + E_2
\end{bmatrix}\rbrace |=\\
 |det( M+ E_1)|  |det \bigg ( N + E_2 -(-N+E_2)(M+E_1)^{-1}(M+E_1)\bigg ) |  =\\
  |det\lbrace M+ E_1\rbrace|  |det\lbrace 2N\rbrace |$.
  \\
It follows that , 
\\
\begin{equation}
\label{equation0}
  | det\lbrace \begin{bmatrix}
g(k,1) &- c\\
v^T &     1
\end{bmatrix}\rbrace |  =2^k|det\lbrace M+ E_1\rbrace|  |det\lbrace N\rbrace | 
\end{equation}
We know
\begin{equation}
\label{equation1}
 |det(N)| \leq \Theta_k
 \end{equation}  since $N$ is an element of $\mathcal{M}_k\lbrace -1,  1\rbrace $. Now let $a = \begin{bmatrix}
 c_1 & \cdots & c_k
 \end{bmatrix}^T$ then we have
\begin{equation}
\label{equation2}
 |det\lbrace M+ E_1\rbrace| = |det \lbrace 
  M + \begin{bmatrix}
c_1r^T\\
\vdots \\
c_{k}r^T
\end{bmatrix}\rbrace | = det \lbrace   \begin{bmatrix}
M &- a\\
r^T &     1
\end{bmatrix}  \rbrace  \leq \Theta_{k+1}
\end{equation}.

 since   $\begin{bmatrix}
M &- a\\
v_1^T &     1
\end{bmatrix} $ is an element of $\mathcal{M}_{k+1}\lbrace -1, + 1\rbrace$. 
From equation (\ref{equation0}),   equation (\ref{equation1}) , and equation (\ref{equation2}) we deduce that
$ | det\lbrace \begin{bmatrix}
g(k,1) &- c\\
v^T &     1
\end{bmatrix}\rbrace |  \leq 2^k\Theta_k\Theta_{k+1}  $
\item In a similar manner one can easily verify that  $ | det\lbrace \begin{bmatrix}
g(k,1) &- c\\
v^T &     1
\end{bmatrix}\rbrace |  \leq 2^k\Theta_k\Theta_{k+1}  $ for $ v^T = \begin{bmatrix}
-r^T & r^T
\end{bmatrix}$
\end{enumerate}
\end{proof}
In Theorem (\ref{theorem4}) we showed the determinant of any saturated design matrix  in  $\mathcal{G}_1(k,1)$ is bounded above by $2^k\Theta_k\Theta_{k+1}$.  Therefore if we can  construct an element of   $\mathcal{G}_1(k,1)$ for which the absolute value of the determinant is $2^k\Theta_k\Theta_{k+1}$ then that element is a saturated  d-optimal design matrix in  $\mathcal{G}_1(k,1)$.  
 In the corollary below we give an element of  $\mathcal{G}_1(k,1)$  for which the absolute value of the determinant is $2^k\Theta_k\Theta_{k+1}$. 

\begin{corollary}
\label{corol3}
Let $\begin{bmatrix}
M_0 &- c_0\\
r_0^T &     1
\end{bmatrix} $ be a matrix  in $\mathcal{M}_{k+1}\lbrace -1, 1\rbrace $ with maximal absolute value of determinants $\Theta_{k+1}$,   and $N_0$  a  matrix  in $\mathcal{M}_{k}\lbrace -1, 1\rbrace $ with maximal absolute value of determinant $\Theta_{k}$,  where $M_0$ is a square  matrix of order $k$,   $c_0$ and $r_0$ are vectors of length $k$. Furthermore let $ v_0^T = \begin{bmatrix}
r_0^T & r_0^T
\end{bmatrix} $, $ \tilde{c}= \begin{bmatrix}
c_0^T & c_0^T
\end{bmatrix}^T $ and  $g_0(k,1) = \begin{bmatrix}
M_0 & M_0\\
-N_0&   N_0
\end{bmatrix} $ .
Then the design matrix  $\begin{bmatrix}
g_0(k,1) &- c*\\
v_0^T  &     1
\end{bmatrix} $ is a d-optimal design in $\mathcal{G}_1(k,1)$. Furthermore 
 $|det\Bigg (\begin{bmatrix}
g_0(k,1) &- \tilde{c}\\
v_0^T  &     1
\end{bmatrix}\Bigg )| = 2^k\Theta_k\Theta_{k+1} $
\end{corollary}
\begin{proof}
It is not hard to see that  the matrix $\begin{bmatrix}
g_0(k,1) &- \tilde{c}\\
v_0^T  &     1
\end{bmatrix} $ satisfies the form of an element of $\mathcal{G}_1(k,1)$. We  show that the absolute value of its determinant attains $2^k\Theta_k\Theta_{k+1}$. The proof is similar to the proof of Theorem (\ref{theorem4}).  
 From Equation (\ref{equation0} ) in the proof of Theorem (\ref{theorem4}) we have :
 
 \begin{equation}
\label{equation3}
  | det\lbrace \begin{bmatrix}
g_0(k,1) &- \tilde{c}\\
v^T &     1
\end{bmatrix}\rbrace |  =2^k|det\lbrace M_0 + \tilde{E_1}\rbrace|  |det\lbrace N_0\rbrace | ,
\end{equation}
where $\tilde{E_1} =\begin{bmatrix}
\tilde{c}_1r_0^T \\
\vdots
\\
\tilde{c}_kr_0^T
\end{bmatrix} $ and $c_0 =\begin{bmatrix}
\tilde{c}_1 \\
\vdots
\\
\tilde{c}_k
\end{bmatrix} $.
We know that  $|det( N_0) | =\Theta_{k}$ and also \\
\begin{equation}
\label{equation4}
|det( M_0 + \tilde{E_1})|  = |det \Bigg (
  M_0 + \begin{bmatrix}
\tilde{c}_1r_0^T \\
\vdots
\\
\tilde{c}_kr_0^T
\end{bmatrix}\Bigg ) | = |det \Big (   \begin{bmatrix}
M_0 &- c_0\\
r_0^T &     1
\end{bmatrix}  \Big ) | = \Theta_{k+1}
\end{equation}.
It follows that  $|det\Bigg (\begin{bmatrix}
g_0(k,1) &- \tilde{c}\\
v_0^T  &     1
\end{bmatrix}\Bigg )| = 2^k\Theta_k\Theta_{k+1} $.
The proof is complete.
\end{proof}
 \subsection{Algorithm for the construction of an element of $\mathcal{G}_1(k,1)$}
 \label{goptimal2}
 We use Theorem (\ref{theorem4}) and Corollary (\ref{corol3}) to develop  an algorithm for the construction of  saturated and  d-optimal saturated design matrices in $\mathcal{G}_1(k,1)$.
 \begin{itemize}
 \item \textbf{Step 1} : Construct/Select two matrices $M_{k+1}$  and $N_k$ from $\mathcal{M}_{k+1}\lbrace -1,1\rbrace $ and $\mathcal{M}_{k}\lbrace -1,1\rbrace $ respectively (For  a d-optimal design select the matrix $M_{k+1}$ and $N_k$ with maximal absolute value of determinant in $\mathcal{M}_{k+1}\lbrace -1,1\rbrace $ and $\mathcal{M}_{k}\lbrace -1,1\rbrace $ respectively.)
 \item \textbf{Step 2}: If the entry of $[M_{k+1}]_{k+1, k+1} = -1$,  then multiply the last column of $M_{k+1}$ by $-1$. After that write $M_{k+1}$ on the form $M_{k+1} =\begin{bmatrix} 
  M_k  & -c_0\\
  
 v_0^T  & 1
\end{bmatrix} $, where $M_k$ is an element of $\mathcal{M}_{k}\lbrace -1,1\rbrace $ and $c_0$ and $v_0$ are vectors of length $k$.  Now set $v := \begin{bmatrix} 
v_0^T & v_0^T
\end{bmatrix}^T$ , $c = \begin{bmatrix} 
c_0^T & c_0^T
\end{bmatrix}^T $ , $v = \begin{bmatrix} 
v_0^T & v_0^T
\end{bmatrix}^T$ and $g_0(k,1) = \begin{bmatrix} 
M_k & M_k \\
-N_k & N_k
\end{bmatrix}$
\item \textbf{Step 3}: The matrix  $\begin{bmatrix} 
g_0(k,1) & -c\\
-v^T & 1
\end{bmatrix}$ is a saturated design matrix in $\mathcal{G}_1(k,1)$. It is a d-optimal saturated design matrix if $M_{k+1}$ and $M_k$ are both d-optimal design matrices in $\mathcal{M}_{k+1}\lbrace -1,1\rbrace $ and $\mathcal{M}_{k}\lbrace -1,1\rbrace $,  respectively. 

 \end{itemize}

\section{ Discussion about $\mathcal{G}_1(k,1)$  and $\mathcal{G}(k,1)$ :  local  and global maximal determinants  and   local upper bounds  }
We have shown so far that  maximal determinants of elements in $\mathcal{G}_1(k,1)$ and $\mathcal{G}(k,1)$ are $2^k\Theta_k\Theta_{k+1}$  and $2^k\Theta_k^2$ respectively.  Thus constructing d-optimal design matrices in $\mathcal{G}_1(k,1)$ and $\mathcal{G}(k,1)$ is only possible if one knows how to construct global d-optimal design matrices  in both $\mathcal{M}_k\lbrace -1, 1\rbrace $ and $\mathcal{M}_{k+1}\lbrace -1, 1\rbrace $. As we discussed in the introduction the construction of the d-optimal design matrix in $\mathcal{M}_k\lbrace -1, 1\rbrace $ and finding a tight upper bound for the determinant of elements in $\mathcal{M}_k \lbrace -1, 1\rbrace$ are  not  easy problems. There has been a lot of ongoing research on the topic for the past hundred years and yet  a lot still has to be done. We use the Ehlich's, Barba's,  Wojtas' and Hadamard's  determinant  upper bounds discussed in the introduction to deduce  determinant upper bounds for elements in $\mathcal{G}_1(k,1)$ and $\mathcal{G}_1(k,1)$ that we display in  Table (\ref{bound1}) and Table (\ref{bound2}). It is worth pointing out that the upper bounds displayed in those tables  pertain  only to  matrices in  $\mathcal{G}_1(k,1)$ and $\mathcal{G}_1(k,1)$  which are of order $2k+1$ and $2k$,  respectively.  We shall therefore  refer to the upper  bounds  for matrices in $\mathcal{M}_{2k+1}\lbrace -1, +1 \rbrace$ and $\mathcal{M}_{2k}\lbrace -1, +1 \rbrace$ as global upper bounds and call the ones in Table (\ref{bound1}) and Table (\ref{bound2}) as local upper bounds  in  $\mathcal{G}_1(k,1)$ and $\mathcal{G}_1(k,1)$. 
We use the algorithms developed in section (\ref{goptimal1}) and Section (\ref{goptimal2}) to \\ construct  four d-optimal design matrices $g^*(5,1)$, $g^*(15,1)$, $g^*(16,1)$, $g_1^*(15,1)$  in $\mathcal{G}(5,1)$ ,   $\mathcal{G}(16,1)$,  $\mathcal{G}(15,1)$  and $\mathcal{G}_1(15,1)$,  repectively. Following the algorithm in Section (\ref{goptimal1})  the construction of $g^*(5,1)$  requires an element of $\mathcal{M}_5\lbrace -1, 1\rbrace $ with maximal absolute value of determinant. The computation is straightforward using R programming language.  For  $g^*(15,1)$ we  need a maximal  absolute value determinant element in $\mathcal{M}_{15}\lbrace -1, 1\rbrace$ which is computationally difficult without a clever algorithm . Fortunately  the matrix $M^*_{15}$ displayed in Figure (\ref{optimal15}) has maximal absolute value of determinant in $\mathcal{M}_{15}\lbrace -1, 1\rbrace$. The construction of $M^*_{15}$ is due to  Smith [\cite{smith}], Cohn [\cite{cohn1}] , [\cite{cohn2}] and Orrick [\cite{orrick1}]. The matrix $M^{*+}_{15}$ displayed in Figure (\ref{optimal15+}) is a normalized version of the matrix $M^*_{15}$.  The matrix $M^{*-}_{15}$ displayed in Figure (\ref{optimal15-}) is just  $- M^{*+}_{15}$. We use the matrices $M^{*+}_{15}$  and $M^{*-}_{15}$ to construct the d-optimal design matrix $g^*(15,1)$ shown in Figure (\ref{optimal15k1}). The design matrix shown in Figure (\ref{optimal16}) is the d-optimal saturated design matrix $g^*(16,1)$ constructed using a Hadamard matrix $H_{16}$  of order $16$ displayed in Figure (\ref{hadamard1}). Following  the algorithm in Section (\ref{goptimal2}) we use $H_{16}$ , $M^{*+}_{15}$ and $M^{*-}_{15}$ to construct the d-optimal saturated design $g_1^*(15,1)$ that we display in Figure (\ref{optimal1516k}). 
\\
\\
In table (\ref{dettable}) we compare the maximal determinants attained by elements in $\mathcal{G}(5,1)$ , $\mathcal{G}(15,1)$,  $\mathcal{G}(16,1)$ and  $\mathcal{G}_1(15,1)$ with their respective  local upper bounds. As we can see from the table  $\mathcal{G}(5,1)$ and $\mathcal{G}(16,1)$ attain $100\%$ of their corresponding local upper bounds. Thus the local upper bounds are  tight for  $\mathcal{G}(5,1)$ and $\mathcal{G}(16,1)$. The maximal determinant for $\mathcal{G}(15,1)$  attains $94\%$  of its local upper bound. That of $\mathcal{G}_1(15,1)$ attains $97\%$ of its local upper bound.  On the other hand the maximal determinants of $\mathcal{G}(5,1)$ and $\mathcal{G}(16,1)$ both attain $100\%$ of the global maximal determinant in $\mathcal{M}_{10}\lbrace -1, 1\rbrace$ and  $\mathcal{M}_{32}\lbrace -1, 1\rbrace$,  respectively. The maximal determinants of $\mathcal{G}(15,1)$ and $\mathcal{G}_1(5,1)$  only attain $54\%$ and $72\%$ of their global maximal determinants in $\mathcal{M}_{30}\lbrace -1, 1\rbrace$  and $\mathcal{M}_{31}\lbrace -1, 1\rbrace$,  respectively. It is worth pointing out the global maximal determinant in $\mathcal{M}_{30}\lbrace -1, 1\rbrace$  is due to Ehlich [\cite{Ehlich1}].  The maximal determinant matrix in  $\mathcal{M}_{32}\lbrace -1, 1\rbrace$  is just a Hadamard matrix which can easily be constructed.  The maximal determinant value $25515\times 2^{14}$ for matrices in $\mathcal{M}_{15}\lbrace -1, 1\rbrace$ was found by Orrick [\cite{orrick1}].  The value $2^{30}\times 784 \times 7^{13}$ is the determinant of a matrix in $\mathcal{M}_{31}\lbrace -1, 1\rbrace$ reported by Hiroki Tamura on  August $26^{th}$, 2005 to  the website \begin{verbatim}
http://www.indiana.edu/~maxdet/d31.html
\end{verbatim} . Though the above value  has not been proved to be optimal,  it is known to be the highest determinant value of an element of $\mathcal{M}_{31}\lbrace -1, 1\rbrace$ found so far. All the other determinants in the table are easily deduced from our theorems and the determinants reported.

\begin{table}
\begin{center}
\[
\begin{cases}
|det\lbrace g_1(k \text{,}1)\rbrace | \leq 2^k k^k\sqrt{(2k+1)}\\
 \quad \text{for } \quad  k \equiv 0 \pmod{4}  \\
 \\
|det\lbrace g_1(k \text{,}1)\rbrace | \leq 2^k (2k)(k-1)\sqrt{ (k-1)^{k-1}(2k-1) }
\\
 \quad  \text{for } \quad k \equiv 1 \pmod{4} \\
 \\
|det\lbrace g_1(k \text{,}1)\rbrace | \leq 2^k (2k-2)(k-2)\sqrt{ (k-2)^{(k+1-s)}(k-2 + 4r)^u(k + 2 + 4r)^v(1-\frac{ur}{k-2 + 4r} - \frac{v(r + 1)}{k + 2 + 4r}  )}
\\
  \quad \text{for } \quad k \equiv 2 \pmod{4} \\
   \\
 |det\lbrace g_1(k \text{,}1)\rbrace | \leq 2^k(k+1)^{\frac{k+1}{2}} \sqrt{ (k-3)^{(k-s)}(k-3 + 4r)^u(k + 1 + 4r)^v(1-\frac{ur}{k-3 + 4r} - \frac{v(r + 1)}{k + 1 + 4r}  )} 
 \\
  \quad  \text{for } \quad k \equiv 3 \pmod{4} 

\end{cases}
\]
\end{center}

\caption{Local upper bounds for the determinant of elements in  $\mathcal{G}_1(k,1)$}
\label{bound1}
\end{table}

\begin{table}
\begin{center}
\[
\begin{cases}
|det\lbrace g(k \text{,}1)\rbrace | \leq 2^k k^k\\
 \quad \text{for } \quad  k \equiv 0 \pmod{4}  \\
 \\
|det\lbrace g(k \text{,}1)\rbrace | \leq 2^k (k-1)^{k-1}(2k-1)
\\
 \quad  \text{for } \quad k \equiv 1 \pmod{4} \\
 \\
|det\lbrace g(k \text{,}1)\rbrace | \leq 2^k (2k-2)^2(k-2)^2
\\
  \quad \text{for } \quad k \equiv 2 \pmod{4} \\
   \\
 |det\lbrace g(k \text{,}1)\rbrace | \leq 2^k  (k-3)^{(k-s)}(k-3 + 4r)^u(k + 1 + 4r)^v(1-\frac{ur}{k-3 + 4r} - \frac{v(r + 1)}{k + 1 + 4r}  )
 \\
  \quad  \text{for } \quad k \equiv 3 \pmod{4} 

\end{cases}
\]
\end{center}
\caption{Local upper bounds for the determinant of elements in  $\mathcal{G}(k,1)$}
\label{bound2}
\end{table}

\begin{table}
\begin{center}
  \footnotesize
\setlength{\arraycolsep}{2pt} 
\medmuskip = 0.5mu 
 \begin{tabular}{|c c c c c|} 
 \hline
   Set of saturated design matrices & $\mathcal{G}(5,1)$ & $\mathcal{G}(15,1)$ & $\mathcal{G}(16,1)$ & $\mathcal{G}_1(15,1)$   \\ [0.5ex] 
 \hline\hline
$ p$ (order of the matrices ) & $10$ & $30$ &$32$ & $31$ \\
 \\
 
  Local maximal determinants & $2^5\times 48^2$ & $2^{15}\times 25515^2\times 2^{28}$  & $2^{16}\times 16^{16}$  &$2^{15}\times 25515\times 2^{14}\times 16^{8}$ \\ 
 \\
\% of   local upperbounds attained & $100$ & $94.23$ & $100$ & $97.07$\\
 \\
 Global maximal determinants  of order $p$ & $18\times 2^{12}$ & $203\times 2^{29}\times 7^{13}$ & $2^{31}\times 16 \times 8^{15} $ & $2^{30}\times 784 \times 7^{13} ??$\\
 \\
  \% of global determinants attained  &  $100$ & $54.23$  &$100$ &$72.13$  \\
 \hline
\end{tabular}
\caption{Upper bounds, local and global determinants comparison for  $\mathcal{G}(5,1)$ , $\mathcal{G}(15,1)$ , $\mathcal{G}(16,1)$ and  $\mathcal{G}_1(15,1)$  }
\label{dettable}
\end{center}
\end{table}
\newpage

 \section{Concluding remarks}
 The construction of saturated design matrices for two level factorial experiment have gained a lot of interest over a long period of time by both mathematicians and statisticians. In general mathematicians are interested in  finding a matrix with maximal determinant in $\mathcal{M}_k\lbrace -1, 1\rbrace $,  as well as  investigating the spectrum of the determinant function which is the set of the value(s) taken by  the $|det(D_k)|/2^{k-1}$ for $D_k$ element of $\mathcal{M}_k\lbrace -1, 1\rbrace $ . Thus  numerous papers have been written about the classification of saturated design matrices of fixed order  via the spectrum of the determinant function.  The spectra of the determinant function $S_k$ for $\lbrace -1, +1\rbrace$-matrices  of order $k$ are  well known in the literature for order up to $11$. The spectrum of order $k=8$ is due to Metropolis, et al.  [\cite{metropolis}]. For $k = 9$ and $k=10$,  the spectra were computed by \v{Z}ivkovi\'c [\cite{zivkovic}] and the spectrum for $k = 11$ is due to Orrick [\cite{orrick1}].
Furthermore  many other papers have investigated d-optimal saturated design matrices for a  fix order.  Orrick [\cite{orrick1}] constructed a  d-optimal design matrix of order $15$.  T. Chadjipantelis, et al.  [\cite{chad}] came up with a d-optimal design of order $21$. The the d-optimal design matrix discussed by these papers is a matrix  with  maximal  absolute value  of the  determinant  in  $\mathcal{M}_k\lbrace -1, 1 \rbrace $. The design statisticians on the other hand  are not only interested in the global d-optimal design matrices in $\mathcal{M}_k\lbrace -1, 1 \rbrace $ but also they are interested in the local d-optimal design matrices that satisfy certain restrictions on the columns of matrices in 
$\mathcal{M}_k\lbrace -1, 1 \rbrace $. In fact more than often It is desirable for  design statisticians to find a d-optimal design matrix to estimate the mean, the main effects and a selected number of two-factor interactions. The restriction imposed by the interactions on the columns of saturated design matrices makes it  impossible to  construct a saturated design matrix  that achieves the maximal determinant in  $\mathcal{M}_k\lbrace -1, 1 \rbrace $ under certain conditions. The work we did in the current  paper is a good illustration.  We showed that  the construction of  saturated d-optimal design matrices in  $\mathcal{G}_1(k,1)$  and $\mathcal{G}(k,1)$  is equivalent to finding   matrices with maximal determinant   in $\mathcal{M}_k\lbrace -1, 1\rbrace$ and $\mathcal{M}_{k+1}\lbrace -1, 1\rbrace$. Thus this problem is just as hard as the Hadamard determinant problem discussed in the introduction.  

\section{Acknowledgements}
This work is partially supported by the US National Science Foundation(NSF Grant 1809681).


\bibliography{uictman}

\newpage

\begin{center}
\begin{figure}
\includegraphics[trim= {0, 10cm, 0, 0} , clip, scale =0.8]{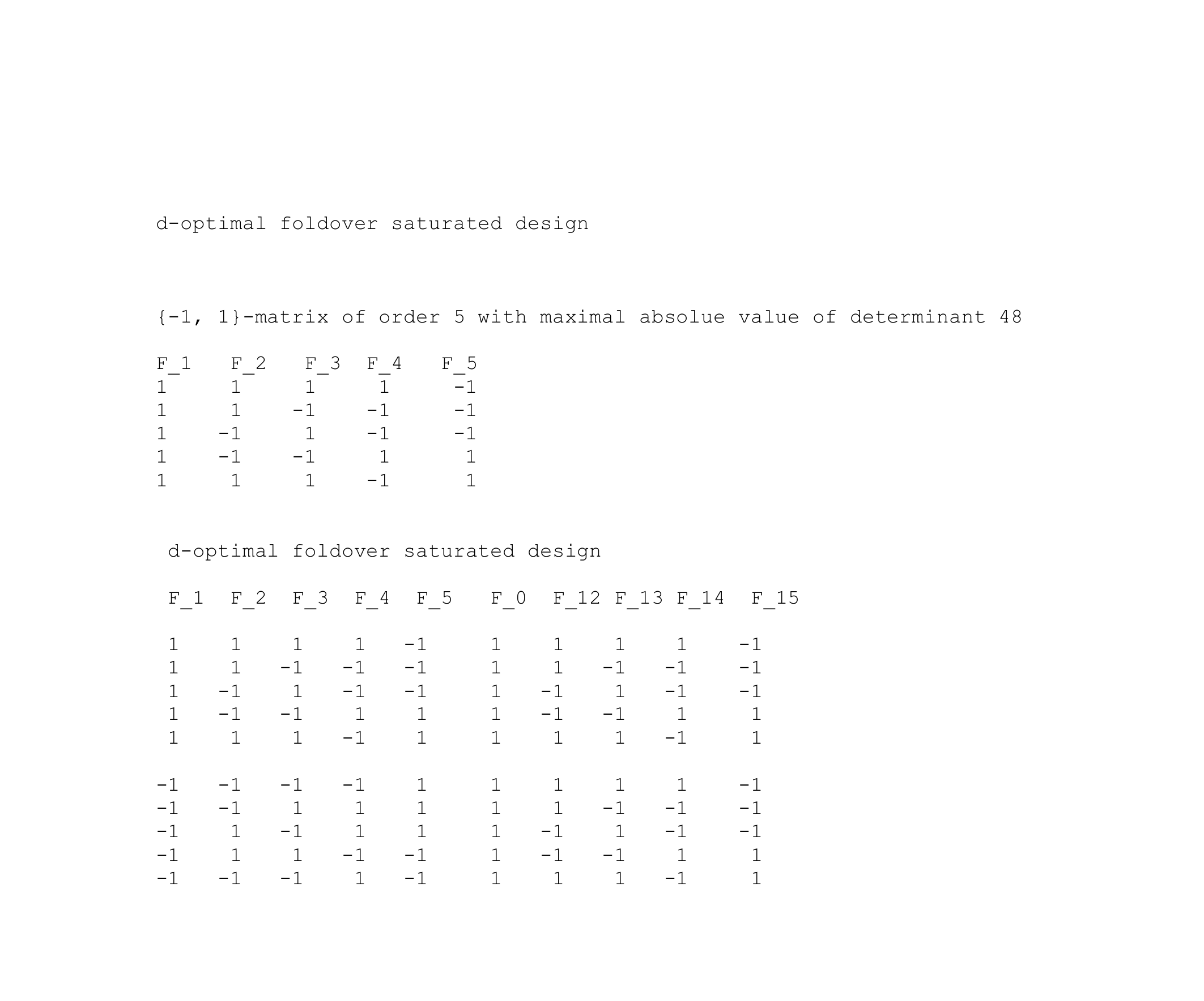}
\caption{d-optimal design for $F_0 ;F_1; F_2; F_3; F_4; F_5;  F_{12};  F_{13}; F_{14};   \text{and}   F_{15}$}
\label{doptimal5}
\end{figure}
\end{center}
\begin{figure}
\footnotesize
\setlength{\arraycolsep}{2.5pt} 
\medmuskip = 1mu 
$
M^*_{15}=\begin{bmatrix}
- & - & -& - & + & + & +& + & + & - &+& + & + & + & -\\
-&-&-&+&-&+&+&-&+&+&+&+&-&+&+\\
-&-&-&+&+&-&+&+&-&+&+&-&+&+&+\\
-&+&+&-&-&+&+&+&-&+&+&-&-&-&-\\
+&+&-&+&-&-&+&+&+&-&+&-&-&-&-\\
+&-&+&-&+&-&+&-&+&+&+&-&-&-&-\\
+&+&+&+&+&+&-&+&+&+&+&-&-&+&-\\
-&+&+&-&-&-&-&-&+&-&+&-&+&+&+ \\
+&+&+&-&-&-&+&+&+&+&-&+&+&+&+\\
+&-&+&-&-&-&-&+&-&-&+&+&-&+&+ \\
+&+&-&-&-&-&-&-&-&+&+&+&+&+&-\\
+&+&-&-&+&+&+&-&-&-&-&-&-&+&+\\
+&-&+&+&-&+&+&-&-&-&-&-&+&+&- \\
-&+&+&+&+&-&+&-&-&-&-&+&-&+&-\\
+&+&+&+&+&+&+&-&-&-&+&+&+&-&+

\end{bmatrix}
$
\caption{Maximal determinant matrix in $\mathcal{M}_{15}\lbrace -1, 1\rbrace $. A result of the work of  Smith [\cite{smith}], Cohn [\cite{cohn1}] , [\cite{cohn2}] and Orrick [\cite{orrick1}] }
\label{optimal15}
\end{figure}

\begin{figure}
\footnotesize
\setlength{\arraycolsep}{2.5pt} 
\medmuskip = 1mu 
$
M^{*+}_{15}=\begin{bmatrix}
+ & + & +& + & - & - & -& - & - & + &-& - & - & - &+\\
+&+&+&-&+&-&-&+&-&-&-&-&+&-&-\\
+&+&+&-&-&+&-&-&+&-&-&+&-&-&-\\
+&-&-&+&+&-&-&-&+&-&-&+&+&+&+\\
+&+&-&+&-&-&+&+&+&-&+&-&-&-&-\\
+&-&+&-&+&-&+&-&+&+&+&-&-&-&-\\
+&+&+&+&+&+&-&+&+&+&+&-&-&+&-\\
+&-&-&+&+&+&+&+&-&+&-&+&-&-&- \\
+&+&+&-&-&-&+&+&+&+&-&+&+&+&+\\
+&-&+&-&-&-&-&+&-&-&+&+&-&+&+ \\
+&+&-&-&-&-&-&-&-&+&+&+&+&+&-\\
+&+&-&-&+&+&+&-&-&-&-&-&-&+&+\\
+&-&+&+&-&+&+&-&-&-&-&-&+&+&- \\
+&-&-&-&-&+&-&+&+&+&+&-&+&-&+\\
+&+&+&+&+&+&+&-&-&-&+&+&+&-&+

\end{bmatrix}
$
\caption{ Normalized maximal determinant matrix in $\mathcal{M}_{15}\lbrace -1, 1\rbrace $ obtained from $M^*_{15}$ }
\label{optimal15+}
\end{figure}

\begin{figure}
\footnotesize
\setlength{\arraycolsep}{2.5pt} 
\medmuskip = 1mu 
$
M^{*-}_{15}=\begin{bmatrix}
- & - & -& - & + & + & +& + & + & - &+& + & + & + & -\\
-&-&-&+&-&+&+&-&+&+&+&+&-&+&+\\
-&-&-&+&+&-&+&+&-&+&+&-&+&+&+\\
-&+&+&-&-&+&+&+&-&+&+&-&-&-&-\\
-&-&+&-&+&+&-&-&-&+&-&+&+&+&+\\
-&+&-&+&-&+&-&+&-&-&-&+&+&+&+\\
-&-&-&-&-&-&+&-&-&-&-&+&+&-&+\\
-&+&+&-&-&-&-&-&+&-&+&-&+&+&+ \\
-&-&-&+&+&+&-&-&-&-&+&-&-&-&-\\
-&+&-&+&+&+&+&-&+&+&-&-&+&-&- \\
-&-&+&+&+&+&+&+&+& -&-&-&-&-&+\\
-&-&+&+&-&-&-&+&+&+&+&+&+&-&-\\
-&+&-&-&+&-&-&+&+&+&+&+&-&-&+ \\
-&+&+&+&+&-&+&-&-&-&-&+&-&+&-\\
-&-&-&-&-&-&-&+&+&+&-&-&-&+&-

\end{bmatrix}
$
\caption{ The opposite matrix of $M^{*+}_{15}$ }
\label{optimal15-}
\end{figure}

\begin{figure}
\footnotesize
\setlength{\arraycolsep}{2.5pt} 
\medmuskip = 1mu 
$
H_{16}=\begin{bmatrix}
+&+&+&+&+&+&+&+&+&+&+&+&+&+&+&+\\
+&-&+&-&+&-&+&-&+&-&+&-&+&-&+&- \\
+&+&-&-&+&+&-&-&+&+&-&-&+&+&-&-\\
+&-&-&+&+&-&-&+&+&-&-&+&+&-&-&+\\
+&+&+&+&-&-&-&-&+&+&+&+&-&-&-&-\\
+&-&+&-&-&+&-&+&+&-&+&-&-&+&-&+\\
+&+&-&-&-&-&+&+&+&+&-&-&-&-&+&+\\
+&-&-&+&-&+&+&-&+&-&-&+&-&+&+&-\\
+&+&+&+&+&+&+&+&-&-&-&-&-&-&-&-\\
+&-&+&-&+&-&+&-&-&+&-&+&-&+&-&+ \\
+&+&-&-&+&+&-&-&-&-&+&+&-&-&+&+\\
+&-&-&+&+&-&-&+&-&+&+&-&-&+&+&-\\
+&+&+&+&-&-&-&-&-&-&-&-&+&+&+&+ \\
+&-&+&-&-&+&-&+&-&+&-&+&+&-&+&- \\
+&+&-&-&-&-&+&+&-&-&+&+&+&+&-&- \\
+&-&-&+&-&+&+&-&-&+&+&-&+&-&-&+
\end{bmatrix}
$
\caption{Hadamard matrix of order 16}
\label{hadamard1}
\end{figure}

 \begin{figure}
  \footnotesize
\setlength{\arraycolsep}{2pt} 
\medmuskip = 0.5mu 

  \begin{align*}
  \begin{matrix}
+ & + & +& + & - & - & -& - & - & + &-& - & - & - &+\\
+&+&+&-&+&-&-&+&-&-&-&-&+&-&-\\
+&+&+&-&-&+&-&-&+&-&-&+&-&-&-\\
+&-&-&+&+&-&-&-&+&-&-&+&+&+&+\\
+&+&-&+&-&-&+&+&+&-&+&-&-&-&-\\
+&-&+&-&+&-&+&-&+&+&+&-&-&-&-\\
+&+&+&+&+&+&-&+&+&+&+&-&-&+&-\\
+&-&-&+&+&+&+&+&-&+&-&+&-&-&- \\
+&+&+&-&-&-&+&+&+&+&-&+&+&+&+\\
+&-&+&-&-&-&-&+&-&-&+&+&-&+&+ \\
+&+&-&-&-&-&-&-&-&+&+&+&+&+&-\\
+&+&-&-&+&+&+&-&-&-&-&-&-&+&+\\
+&-&+&+&-&+&+&-&-&-&-&-&+&+&- \\
+&-&-&-&-&+&-&+&+&+&+&-&+&-&+\\
+&+&+&+&+&+&+&-&-&-&+&+&+&-&+
\end{matrix} &
\begin{matrix}
+ & + & +& + & - & - & -& - & - & + &-& - & - & - &+\\
+&+&+&-&+&-&-&+&-&-&-&-&+&-&-\\
+&+&+&-&-&+&-&-&+&-&-&+&-&-&-\\
+&-&-&+&+&-&-&-&+&-&-&+&+&+&+\\
+&+&-&+&-&-&+&+&+&-&+&-&-&-&-\\
+&-&+&-&+&-&+&-&+&+&+&-&-&-&-\\
+&+&+&+&+&+&-&+&+&+&+&-&-&+&-\\
+&-&-&+&+&+&+&+&-&+&-&+&-&-&- \\
+&+&+&-&-&-&+&+&+&+&-&+&+&+&+\\
+&-&+&-&-&-&-&+&-&-&+&+&-&+&+ \\
+&+&-&-&-&-&-&-&-&+&+&+&+&+&-\\
+&+&-&-&+&+&+&-&-&-&-&-&-&+&+\\
+&-&+&+&-&+&+&-&-&-&-&-&+&+&- \\
+&-&-&-&-&+&-&+&+&+&+&-&+&-&+\\
+&+&+&+&+&+&+&-&-&-&+&+&+&-&+
\end{matrix}\\
\begin{matrix}
- & - & -& - & + & + & +& + & + & - &+& + & + & + & -\\
-&-&-&+&-&+&+&-&+&+&+&+&-&+&+\\
-&-&-&+&+&-&+&+&-&+&+&-&+&+&+\\
-&+&+&-&-&+&+&+&-&+&+&-&-&-&-\\
-&-&+&-&+&+&-&-&-&+&-&+&+&+&+\\
-&+&-&+&-&+&-&+&-&-&-&+&+&+&+\\
-&-&-&-&-&-&+&-&-&-&-&+&+&-&+\\
-&+&+&-&-&-&-&-&+&-&+&-&+&+&+ \\
-&-&-&+&+&+&-&-&-&-&+&-&-&-&-\\
-&+&-&+&+&+&+&-&+&+&-&-&+&-&- \\
-&-&+&+&+&+&+&+&+& -&-&-&-&-&+\\
-&-&+&+&-&-&-&+&+&+&+&+&+&-&-\\
-&+&-&-&+&-&-&+&+&+&+&+&-&-&+ \\
-&+&+&+&+&-&+&-&-&-&-&+&-&+&-\\
-&-&-&-&-&-&-&+&+&+&-&-&-&+&-
\end{matrix}&
\begin{matrix}
+ & + & +& + & - & - & -& - & - & + &-& - & - & - &+\\
+&+&+&-&+&-&-&+&-&-&-&-&+&-&-\\
+&+&+&-&-&+&-&-&+&-&-&+&-&-&-\\
+&-&-&+&+&-&-&-&+&-&-&+&+&+&+\\
+&+&-&+&-&-&+&+&+&-&+&-&-&-&-\\
+&-&+&-&+&-&+&-&+&+&+&-&-&-&-\\
+&+&+&+&+&+&-&+&+&+&+&-&-&+&-\\
+&-&-&+&+&+&+&+&-&+&-&+&-&-&- \\
+&+&+&-&-&-&+&+&+&+&-&+&+&+&+\\
+&-&+&-&-&-&-&+&-&-&+&+&-&+&+ \\
+&+&-&-&-&-&-&-&-&+&+&+&+&+&-\\
+&+&-&-&+&+&+&-&-&-&-&-&-&+&+\\
+&-&+&+&-&+&+&-&-&-&-&-&+&+&- \\
+&-&-&-&-&+&-&+&+&+&+&-&+&-&+\\
+&+&+&+&+&+&+&-&-&-&+&+&+&-&+
\end{matrix}  
 \end{align*}
 \caption{
 Saturated d-optimal design matrix for  
 \newline
 $F_1;  F_2 ;  F_3 ; F_4 ; F_5 ;  F_6 ; F_7 ; F_8 ;  F_9 ;F_{10} ; F_{11} ;F_{12} ; F_{13} ; F_{14} ; F_{15}$;
 \newline
  $F_0 ; F_{1,2} ;   F_{1,3} ; F_{1,4}; F_{1,5} ;  F_{1,6};  F_{1,7} ; F_{1,8} ;F_{1,9} ; F_{1,10} ; F_{1,11} ;
   F_{1,12} ; F_{1,13} ; F_{1,14} ; \text{and} F_{1,15}   $  }
   \label{optimal15k1}
  \end{figure}

 \newpage
 \begin{figure}
  \footnotesize
\setlength{\arraycolsep}{2pt} 
\medmuskip = 0.5mu 

  \begin{align*}
  \begin{matrix}
+&+&+&+&+&+&+&+&+&+&+&+&+&+&+\\
+&-&+&-&+&-&+&-&+&-&+&-&+&-&+ \\
+&+&-&-&+&+&-&-&+&+&-&-&+&+&-\\
+&-&-&+&+&-&-&+&+&-&-&+&+&-&-\\
+&+&+&+&-&-&-&-&+&+&+&+&-&-&-\\
+&-&+&-&-&+&-&+&+&-&+&-&-&+&-\\
+&+&-&-&-&-&+&+&+&+&-&-&-&-&+\\
+&-&-&+&-&+&+&-&+&-&-&+&-&+&+\\
+&+&+&+&+&+&+&+&-&-&-&-&-&-&-\\
+&-&+&-&+&-&+&-&-&+&-&+&-&+&- \\
+&+&-&-&+&+&-&-&-&-&+&+&-&-&+\\
+&-&-&+&+&-&-&+&-&+&+&-&-&+&+\\
+&+&+&+&-&-&-&-&-&-&-&-&+&+&+ \\
+&-&+&-&-&+&-&+&-&+&-&+&+&-&+ \\
+&+&-&-&-&-&+&+&-&-&+&+&+&+&- 
\end{matrix}
&\begin{matrix}
+&+&+&+&+&+&+&+&+&+&+&+&+&+&+&+\\
+&-&+&-&+&-&+&-&+&-&+&-&+&-&+&- \\
+&+&-&-&+&+&-&-&+&+&-&-&+&+&-&-\\
+&-&-&+&+&-&-&+&+&-&-&+&+&-&-&+\\
+&+&+&+&-&-&-&-&+&+&+&+&-&-&-&-\\
+&-&+&-&-&+&-&+&+&-&+&-&-&+&-&+\\
+&+&-&-&-&-&+&+&+&+&-&-&-&-&+&+\\
+&-&-&+&-&+&+&-&+&-&-&+&-&+&+&-\\
+&+&+&+&+&+&+&+&-&-&-&-&-&-&-&-\\
+&-&+&-&+&-&+&-&-&+&-&+&-&+&-&+ \\
+&+&-&-&+&+&-&-&-&-&+&+&-&-&+&+\\
+&-&-&+&+&-&-&+&-&+&+&-&-&+&+&-\\
+&+&+&+&-&-&-&-&-&-&-&-&+&+&+&+ \\
+&-&+&-&-&+&-&+&-&+&-&+&+&-&+&- \\
+&+&-&-&-&-&+&+&-&-&+&+&+&+&-&- 
\end{matrix}\\
\begin{matrix}
- & - & -& - & + & + & +& + & + & - &+& + & + & + & - \\
-&-&-&+&-&+&+&-&+&+&+&+&-&+&+ \\
-&-&-&+&+&-&+&+&-&+&+&-&+&+&+ \\
-&+&+&-&-&+&+&+&-&+&+&-&-&-&- \\
-&-&+&-&+&+&-&-&-&+&-&+&+&+&+ \\
-&+&-&+&-&+&-&+&-&-&-&+&+&+&+ \\
-&-&-&-&-&-&+&-&-&-&-&+&+&-&+ \\
-&+&+&-&-&-&-&-&+&-&+&-&+&+&+ \\
-&-&-&+&+&+&-&-&-&-&+&-&-&-&- \\
-&+&-&+&+&+&+&-&+&+&-&-&+&-&-  \\
-&-&+&+&+&+&+&+&+& -&-&-&-&-&+\\
-&-&+&+&-&-&-&+&+&+&+&+&+&-&-\\
-&+&-&-&+&-&-&+&+&+&+&+&-&-&+  \\
-&+&+&+&+&-&+&-&-&-&-&+&-&+&-\\
-&-&-&-&-&-&-&+&+&+&-&-&-&+&- 
\end{matrix} &
  \begin{matrix}
+ & + & +& + & - & - & -& - & - & + &-& - & - & - &+ &+ \\
+&+&+&-&+&-&-&+&-&-&-&-&+&-&- &-\\
+&+&+&-&-&+&-&-&+&-&-&+&-&-&-&-\\
+&-&-&+&+&-&-&-&+&-&-&+&+&+&+&+\\
+&+&-&+&-&-&+&+&+&-&+&-&-&-&-&-\\
+&-&+&-&+&-&+&-&+&+&+&-&-&-&-&+\\
+&+&+&+&+&+&-&+&+&+&+&-&-&+&-&+\\
+&-&-&+&+&+&+&+&-&+&-&+&-&-&-&- \\
+&+&+&-&-&-&+&+&+&+&-&+&+&+&+&-\\
+&-&+&-&-&-&-&+&-&-&+&+&-&+&+ &+\\
+&+&-&-&-&-&-&-&-&+&+&+&+&+&-&+\\
+&+&-&-&+&+&+&-&-&-&-&-&-&+&+&-\\
+&-&+&+&-&+&+&-&-&-&-&-&+&+&-&+ \\
+&-&-&-&-&+&-&+&+&+&+&-&+&-&+&-\\
+&+&+&+&+&+&+&-&-&-&+&+&+&-&+&-
  \end{matrix}\\
  \begin{matrix}
  +&-&-&+&-&+&+&-&-&+&+&-&+&-&- 
  \end{matrix}
  &
  \begin{matrix}
  +&-&-&+&-&+&+&-&-&+&+&-&+&-&- & +
  \end{matrix}
 \end{align*}
 \caption{
 Saturated d-optimal design matrix for  
 \newline
 $F_1;  F_2 ;  F_3 ; F_4 ; F_5 ;  F_6 ; F_7 ; F_8 ;  F_9 ;F_{10} ; F_{11} ;F_{12} ; F_{13} ; F_{14} ; F_{15}$;
 \newline
  $F_0 ; F_{1,2} ;   F_{1,3} ; F_{1,4}; F_{1,5} ;  F_{1,6};  F_{1,7} ; F_{1,8} ;F_{1,9} ; F_{1,10} ; F_{1,11} ;
   F_{1,12} ; F_{1,13} ; F_{1,14} ; F_{1,15} ;\text{and}  F^e_1$  }
   \label{optimal1516k}
  \end{figure}

 \newpage
 \begin{figure}
  \footnotesize
\setlength{\arraycolsep}{2pt} 
\medmuskip = 0.5mu 

  \begin{align*}
  \begin{matrix}
+&+&+&+&+&+&+&+&+&+&+&+&+&+&+&+\\
+&-&+&-&+&-&+&-&+&-&+&-&+&-&+&- \\
+&+&-&-&+&+&-&-&+&+&-&-&+&+&-&-\\
+&-&-&+&+&-&-&+&+&-&-&+&+&-&-&+\\
+&+&+&+&-&-&-&-&+&+&+&+&-&-&-&-\\
+&-&+&-&-&+&-&+&+&-&+&-&-&+&-&+\\
+&+&-&-&-&-&+&+&+&+&-&-&-&-&+&+\\
+&-&-&+&-&+&+&-&+&-&-&+&-&+&+&-\\
+&+&+&+&+&+&+&+&-&-&-&-&-&-&-&-\\
+&-&+&-&+&-&+&-&-&+&-&+&-&+&-&+ \\
+&+&-&-&+&+&-&-&-&-&+&+&-&-&+&+\\
+&-&-&+&+&-&-&+&-&+&+&-&-&+&+&-\\
+&+&+&+&-&-&-&-&-&-&-&-&+&+&+&+ \\
+&-&+&-&-&+&-&+&-&+&-&+&+&-&+&- \\
+&+&-&-&-&-&+&+&-&-&+&+&+&+&-&- \\
+&-&-&+&-&+&+&-&-&+&+&-&+&-&-&+
\end{matrix} &\begin{matrix}
+&+&+&+&+&+&+&+&+&+&+&+&+&+&+&+\\
+&-&+&-&+&-&+&-&+&-&+&-&+&-&+&- \\
+&+&-&-&+&+&-&-&+&+&-&-&+&+&-&-\\
+&-&-&+&+&-&-&+&+&-&-&+&+&-&-&+\\
+&+&+&+&-&-&-&-&+&+&+&+&-&-&-&-\\
+&-&+&-&-&+&-&+&+&-&+&-&-&+&-&+\\
+&+&-&-&-&-&+&+&+&+&-&-&-&-&+&+\\
+&-&-&+&-&+&+&-&+&-&-&+&-&+&+&-\\
+&+&+&+&+&+&+&+&-&-&-&-&-&-&-&-\\
+&-&+&-&+&-&+&-&-&+&-&+&-&+&-&+ \\
+&+&-&-&+&+&-&-&-&-&+&+&-&-&+&+\\
+&-&-&+&+&-&-&+&-&+&+&-&-&+&+&-\\
+&+&+&+&-&-&-&-&-&-&-&-&+&+&+&+ \\
+&-&+&-&-&+&-&+&-&+&-&+&+&-&+&- \\
+&+&-&-&-&-&+&+&-&-&+&+&+&+&-&- \\
+&-&-&+&-&+&+&-&-&+&+&-&+&-&-&+
\end{matrix}\\
\begin{matrix}
-&-&-&-&-&-&-&-&-&-&-&-&-&-&-&-\\
-&+&-&+&-&+&-&+&-&+&-&+&-&+&-&+ \\
-&-&+&+&-&-&+&+&-&-&+&+&-&-&+&+\\
-&+&+&-&-&+&+&-&-&+&+&-&-&+&+&-\\
-&-&-&-&+&+&+&+&-&-&-&-&+&+&+&+\\
-&+&-&+&+&-&+&-&-&+&-&+&+&-&+&-\\
-&-&+&+&+&+&-&-&-&-&+&+&+&+&-&-\\
-&+&+&-&+&-&-&+&-&+&+&-&+&-&-&+\\
-&-&-&-&-&-&-&-&+&+&+&+&+&+&+&+\\
-&+&-&+&-&+&-&+&+&-&+&-&+&-&+&- \\
-&-&+&+&-&-&+&+&+&+&-&-&+&+&-&-\\
-&+&+&-&-&+&+&-&+&-&-&+&+&-&-&+\\
-&-&-&-&+&+&+&+&+&+&+&+&-&-&-&- \\
-&+&-&+&+&-&+&-&+&-&+&-&-&+&-&+ \\
-&-&+&+&+&+&-&-&+&+&-&-&-&-&+&+ \\
-&+&+&-&+&-&-&+&+&-&-&+&-&+&+&-
\end{matrix} &
\begin{matrix}
+&+&+&+&+&+&+&+&+&+&+&+&+&+&+&+\\
+&-&+&-&+&-&+&-&+&-&+&-&+&-&+&- \\
+&+&-&-&+&+&-&-&+&+&-&-&+&+&-&-\\
+&-&-&+&+&-&-&+&+&-&-&+&+&-&-&+\\
+&+&+&+&-&-&-&-&+&+&+&+&-&-&-&-\\
+&-&+&-&-&+&-&+&+&-&+&-&-&+&-&+\\
+&+&-&-&-&-&+&+&+&+&-&-&-&-&+&+\\
+&-&-&+&-&+&+&-&+&-&-&+&-&+&+&-\\
+&+&+&+&+&+&+&+&-&-&-&-&-&-&-&-\\
+&-&+&-&+&-&+&-&-&+&-&+&-&+&-&+ \\
+&+&-&-&+&+&-&-&-&-&+&+&-&-&+&+\\
+&-&-&+&+&-&-&+&-&+&+&-&-&+&+&-\\
+&+&+&+&-&-&-&-&-&-&-&-&+&+&+&+ \\
+&-&+&-&-&+&-&+&-&+&-&+&+&-&+&- \\
+&+&-&-&-&-&+&+&-&-&+&+&+&+&-&- \\
+&-&-&+&-&+&+&-&-&+&+&-&+&-&-&+
\end{matrix}
  \end{align*}
 
 \caption{
 Saturated d-optimal design matrix for  
 \newline
 $F_1;  F_2 ;  F_3 ; F_4 ; F_5 ;  F_6 ; F_7 ; F_8 ;  F_9 ;F_{10} ; F_{11} ;F_{12} ; F_{13} ; F_{14} ; F_{15} ; F_{16}$;
 \newline
  $F_0 ; F_{1,2} ;   F_{1,3} ; F_{1,4}; F_{1,5} ;  F_{1,6};  F_{1,7} ; F_{1,8} ;F_{1,9} ; F_{1,10} ; F_{1,11} ;
   F_{1,12} ; F_{1,13} ; F_{1,14} ; F_{1,15} ;\text{and}  F_{1, 16}$  }
   \label{optimal16}
  \end{figure}

\end{document}